\documentclass{IEEEtran}
\usepackage{amsmath,amssymb,amsfonts,amsthm}
\usepackage{graphicx,color}
\usepackage{hyperref}
\usepackage{algorithm}
\usepackage{varioref}
\usepackage{color}
\graphicspath{{./Figures/}}

\newcommand{\fromto}[2]{\{#1,\dots, #2\}}

\newcommand{\subscr}[2]{#1_{\textup{#2}}}
\newcommand{\supscr}[2]{#1^{\textup{#2}}}
\newcommand{\setdef}[2]{\{#1 \, | \; #2\}}
\newcommand{\map}[3]{#1: #2 \rightarrow #3}

\newcommand{\real}{\mathbb{R}}

\renewcommand{\natural}{\mathbb{N}}

\newcommand{\R}{\mathbb{R}}
\newcommand{\N}{\mathbb{N}}

\newcommand{\argmax}{\operatornamewithlimits{argmax}}
\newcommand{\dst}{\operatornamewithlimits{dst}}
\newcommand{\diam}{\operatornamewithlimits{diam}}

\newcommand{\dmax}{\subscr{d}{max}}
\newcommand{\card}[1]{{|#1|}}

{\theoremstyle{definition}
\newtheorem{example}{Example}

}%
\newtheorem{thm}{Theorem}

\newtheorem{lem}[thm]{Lemma}

\newtheorem{prop}[thm]{Proposition}

\newcommand{\1}{\mathbf{1}}

\newcommand{\be}{\begin{equation}}
\newcommand{\ee}{\end{equation}}

\title{Message Passing Optimization of  Harmonic Influence Centrality}
\author{
Luca Vassio,\thanks{L.~Vassio is with Dipartimento di Ingegneria Meccanica e Aerospaziale, Politecnico di Torino, Corso Duca degli Abruzzi 24, 10128 Torino, Italy. Tel: +39-339-2868999.  \texttt{luca.vassio@polito.it}.}
\and
Fabio Fagnani,\thanks{
F.~Fagnani is with Dipartimento di Scienze Matematiche, Politecnico di Torino, Corso Duca degli Abruzzi 24, 10128 Torino, Italy. Tel: +39-011-0907509.  \texttt{fabio.fagnani@polito.it}.}
\and Paolo Frasca and\thanks{
P.~Frasca is with Department of Applied Mathematics, University of Twente, Drienerlolaan 5, 7522 NB Enschede, the Netherlands. Tel: +31-53-4893406. \texttt{p.frasca@utwente.nl}.}
\and Asuman Ozdaglar\thanks{
A.~Ozaglar is with LIDS, Massachusetts Institute of Technology, 77 Massachusetts Avenue, Cambridge, 02139 Massachusetts. Tel: +1-617-3240058.  \texttt{asuman@mit.edu}.}
}%

\date{Working paper, \today}

\begin{document}
\maketitle

\begin{abstract}
This paper proposes a new measure of node centrality in social networks, the {\em Harmonic Influence Centrality}, which emerges naturally in the study of social influence over networks. Using an intuitive analogy between social and electrical networks, we introduce a distributed {\em message passing} algorithm to compute the Harmonic Influence Centrality of each node. Although its design is based on theoretical results which assume the network to have no cycle, the algorithm can also be successfully applied on general graphs.
\end{abstract}



\section{Introduction}\label{sect:intro}

A key issue in the study of networks is the identification of their most important nodes: the definition of prominence is based on a suitable function of the nodes, called {\em centrality measure}.
The appropriate notion of centrality measure of a node depends on the nature of interactions among the agents situated in the network and the potential decision and control objectives.
In this paper, we define a new measure of centrality,  which we call {\em Harmonic Influence Centrality} (HIC) and which emerges naturally in the context of social influence. We explain why in addition to being descriptively useful, this measure answers questions related to the optimal placement of different agents or opinions in a network with the purpose of swaying average opinion. In large networks approximating real world social networks, computation of centrality measures of all nodes can be a challenging task. In this paper, we present a fully decentralized algorithm based on message passing for computing HIC of all nodes, which converges to the correct values for trees (connected networks with no cycles), but also can be applied to general networks.

Our model of social influence builds on recent work \cite{AC12}, which characterizes opinion dynamics in a network consisting of stubborn agents who hold a fixed opinion equal to zero or one ({\it i.e.}, type zero and type one stubborn agents) and regular agents who hold an opinion $x_i\in [0,1]$ and update it as a weighted average of their opinion and those of their neighbors.
We consider a special case of this model where a fixed subset of the agents are type zero stubborn agents and the rest are regular agents.
We define the HIC of node $\ell$ as the asymptotic value of the average opinion when node $\ell$ switches from a regular agent to a type one stubborn agent. This measure hence captures the
long run influence of node $\ell$  on the average opinion of the social network.

The HIC measure, beside its natural descriptive value, is also the precise answer to the following network decision problem: suppose you would like to have the largest influence on long run average opinion in the network and you have the capacity to change one agent from regular to type one stubborn. Which agent should be picked for this conversion? This question has a precise answer in terms of HIC; the agent with the highest HIC should be picked.

Though the HIC measure is intuitive, its centralized computation in a large network would be challenging in terms of its informational requirements that involve the network topology and the location of type zero stubborn agents. We propose here a decentralized algorithm whereby each agent computes its own HIC based on local information.
The construction of our algorithm uses a novel analogy between social and electrical networks by relating
 the Laplace equation resulting from social influence dynamics to the governing equations of electrical networks. Under this analogy, the asymptotic opinion of regular agent $i$ can be interpreted as the voltage of node $i$ when type zero stubborn agents are kept at voltage zero and type one agents are kept at voltage one. This interpretation allows us to use tricks of electrical circuits and provide a recursive characterization of HIC in trees. Using this characterization, we develop a so called {\em message passing} algorithm for its solution, which converges after at most a number of steps equal to the diameter of the tree.
The algorithm we propose runs in a distributed and parallel way among the nodes, which do not need to know the topology of the whole network, and has a lower cost in terms of number of operations with respect to the centralized algorithm recently proposed in \cite{YA11}.
Although originally designed for trees, our algorithm can be employed in general networks. For regular networks with unitary resistances (corresponding to all agents placing equal weights on opinions of other agents), we show this algorithm converges (but not necessarily to the correct HIC values). Moreover, we show through simulations that the algorithm performs well also on general networks.

\subsection{Related works}
In social science and network science there is a large literature on defining and computing centrality measures \cite{NEF:91}.
Among the most popular definitions, we mention degree centrality, node betweenness, information centrality~\cite{KS-MZ:89}, and Bonacich centrality~\cite{B87}, which is related to the well-studied Google PageRank algorithm. These notions have proven useful in a range of applications but are not universally the appropriate concepts by any means. Our interest in opinion dynamics thus motivates the choice to define a centrality measure for our purposes.
The problem of computing the HIC has previously been solved, via a centralized solution, in~\cite{YA11} for a special case of the opinion dynamics model considered here.

As opposed to centralized algorithms, the interest for distributed algorithms to compute centrality measures has risen more recently. In~\cite{HI-RT:10} a randomized algorithm is used to compute PageRank centrality. In~\cite{WW-CYT:13} distributed algorithms are designed to compute node (and edge) betweenness on trees, but are not suitable for general graphs.

Message passing algorithms have been widely studied and are commonly used in artificial intelligence and information theory and have demonstrated empirical success in numerous applications including low-density parity-check codes, turbo codes, free energy approximation, and satisfiability problems (see {\it e.g.} \cite{P88} \cite{BM05} \cite{W06} \cite{MM-AM:09}). In such algorithms, nodes are thought as objects with computational ability which can send and receive information to their neighbors. Interesting connections between message passing algorithms and electrical networks are discussed in~\cite{POV-HAL:03}. In fact, electrical networks have also been put in connection with social networks, especially because of the notion of resistance distance~\cite{EE-NH:10}.

\subsection{Paper outline}
In Section~\ref{sect:opinion-dynamics}
we review our model of opinion dynamics with stubborn agents and we define our problem of interest, that is the optimal placement of a stubborn agent.
In Section~\ref{sect:electrical} we review basic notions of electrical circuits and explain theirs connection with social networks.
Section~\ref{sect:electrical-trees} is devoted to apply this electrical analogy on tree graphs and to study the optimal placement problem in this important case. Next,
in Section~\ref{sect:trees-mpa} we describe the message passing algorithm to compute the optimal solution on trees, while in Section~\ref{sect:general-mpa} we consider its extension to general graphs: we provide both theoretical results, in Section~\ref{sect:mpa-regular}, and numerical simulations, in Section~\ref{sect:simulations}. Final remarks are given in Section~\ref{sect:conclusions}.

\subsection{Notation}

To avoid possible ambiguities, we here briefly recall some notation and a few basic notions of graph theory which are used throughout this paper. The cardinality of a (finite) set $E$ is denoted by $\card{E}$ and when $E\subset F$ we define its complement as $E^c=\setdef{f\in F}{f\notin E}.$ A square matrix $P$ is said to be nonnegative when its entries $P_{ij}$ are nonnegative, substochastic when it is nonnegative and $\sum_j P_{ij}\le1$ for every row $i$, and stochastic when it is nonnegative and $\sum_j P_{ij}=1$ for every $i$. We denote by $\1$ a vector of whose entries are all 1.
An (undirected) graph $G$ is a pair $(I,E) $ where $I$ is a finite set, whose elements are said to be the \emph{nodes} of $G$ and $E$ is a set of unordered pairs  of nodes called \emph{edges}.
The neighbor set of a node $i\in I$ is defined as $N_i:=\{j \in I |\{i,j\}\in E\}$. The cardinality of the neighbor set $d_i:=|N_i|$ is said to be the {\em degree} of node $i$. A graph in which all nodes have degree $d$ is said to be {\em $d$-regular}. A path in $G$ is a sequence of nodes $\gamma =(j_1,\dots j_s)$ such that $\{j_t,j_{t+1}\}\in E$ for every $t=1,\dots ,s-1$. The path $\gamma$ is said to connect $j_1$ and $j_s$. The path $\gamma$ is said to be simple if $j_h\neq j_k$ for $h\neq k$. A graph is connected if any pair of distinct nodes can be connected by a path (which can be chosen to be simple).  The length of the shortest path between two nodes $i$ and $j$ is said to be the {\em distance} between them, and is denoted as $\dst(i,j)$. Consequently, the {\em diameter} of a connected graph is defined to be
$\diam(G):=\max_{i,j \in I} \{\dst (i,j)\}.$ A tree is a connected graph such that for any pair of distinct nodes there is just one simple path connecting them.
Finally, given a graph $G=(I,E)$ and a subset $J\subseteq I$, the subgraph induced by $J$ is defined as $G_{|J}=(J, E_{|J})$ where $E_{|J}=\{\{i,j\}\in E\,|\, i,j\in J\}$.

\section{Opinion dynamics and stubborn agent placement}\label{sect:opinion-dynamics}
Consider a connected graph $G=(I,E)$. Nodes in $I$ will be thought as agents who can exchange information through the available edges $\{i,j\}\in E$. Each agent $i\in I$ has an opinion $x_i(t)\in\R$ possibly changing in time $t\in\N$.
We assume a splitting
$I=S\cup R$ with the understanding that agents in $S$ are {\it stubborn} agents not changing their opinions while those in $R$ are {\em regular} agents whose opinions modify in time according to the consensus dynamics
$$x_i(t+1)=\sum\limits_{j\in I}Q_{ij}x_j(t)\,,\quad \forall i\in R$$
where $Q_{ij}\geq 0$ for all $i\in R$ and for all $j\in I$ and $\sum_jQ_{ij}$=1 for all $i\in R$. The scalar $Q_{ij}$ represents the relative weight that agent $i$ places on agent $j$'s opinion. We will assume  that $Q$ only uses the available edges in $G$, more precisely, our standing assumption will be that
\begin{equation}\label{adapted} Q_{ij}=0\;\Leftrightarrow \{i,j\}\not\in E\end{equation}
A basic example is obtained by choosing for each regular agent uniform weights along the edges incident to it, {\it i.e.}, $Q_{ij}=d_i^{-1}$ for all $i\in R$ and $\{i,j\}\in E$.
Assembling opinions of regular and stubborn agents in vectors, denoted by $x^R(t)$ and $x^S(t)$, we can rewrite the dynamics in a more compact form as
\begin{equation*} \begin{array}{rcl}x^R(t+1)&=&Q^{11}x^R(t)+Q^{12}x^S(t)\\ x^S(t+1)&=&x^S(t)\end{array}\end{equation*}
where the matrices $Q^{11}$ and $Q^{12}$ are nonnegative matrices of appropriate dimensions.

Using the adaptivity assumption (\ref{adapted}), it is standard to show that $Q^{11}$ is a substochastic asymptotically stable matrix (e.g. spectral radius $<1$). Henceforth, $x^R(t)\to x^R(\infty)$ for $t\to +\infty$ with the limit opinions satisfying the relation
\begin{equation}\label{fixed point}x^R(\infty)=Q^{11}x^R(\infty)+Q^{12}x^S(0)\end{equation}
which is equivalent to
\begin{equation}\label{fixed point2}x^R(\infty) =(I-Q^{11})^{-1}Q^{12}x^S(0)\end{equation}
Notice that $[(I-Q^{11})^{-1}Q^{12}]_{hk}=[\sum_n(Q^{11})^nQ^{12}]_{hk}$ is always non negative and  is nonzero if and only if there exists a path in $G$ connecting the regular agent $h$ to the stubborn agent $k$ and not touching other stubborn agents. Moreover, the fact that $P$ is stochastic easily implies that $\sum_k[(I-Q^{11})^{-1}Q^{12}]_{hk}=1$ for all $h\in R$: asymptotic opinions of regular agents are thus convex combinations of the opinions of stubborn agents.

In this paper we will focus on the situation when $S=S^{0}\cup\{\ell\}$ and $R=I\setminus S$ assuming that $x_i(0)=0$ for all $i\in S^0$ while $x_\ell(0)=1$, {\it i.e.}, there are two types of stubborn agents: one type consisting of those in set $S^0$ that have opinion 0 and the other type consisting of the single agent $\ell$ that has opinion 1. We investigate how to choose $\ell$ in $I\setminus S^0$ in such a way to maximize the influence of opinion $1$ on the limit opinions. More precisely, let us denote as $x_i^{R,\ell}(\infty)$ the asymptotic opinion of the regular agent $i\in R$ under the above stubborn configuration, and define the objective function
\begin{equation}\label{H_definition}H(\ell):=\sum\limits_{i\in R}x_i^{R,\ell}(\infty)\end{equation}
Notice that the subset $R$ itself is actually a function of $\ell$, however we have preferred not to indicate such dependence to avoid too heavy notations.
The \emph{Optimal Stubborn Agent Placement} (OSAP) is then formally defined as
\begin{equation}\label{eq:OSAP}
 \max_{\ell \in I \setminus S^0} H(\ell).
\end{equation}
In this  optimization problem, for any different choice of $\ell$, the block matrices $Q^{11}$ and $Q^{12}$ change and a new matrix inversion $(I-Q^{11})^{-1}$ needs to be performed. Such matrix inversions require global information about the network and are not feasible for a distributed implementation. In this paper, we propose a fully decentralized algorithm for the computation of asymptotic opinions and for solving the optimization problem which is based on exploiting a classical analogy with electrical circuits. Under such analogy, we can interpret $x_i^{R,\ell}(\infty)$ as the voltage at node $i$ when nodes in $S^{0}$ are kept at voltage $0$ while $\ell$ is kept at voltage $1$. This interpretation results to be quite useful as it allows to use typical ``tricks'' of electrical circuits (e.g. parallel and series reduction, glueing).

In order to use the electrical circuit analogy,  we will need to make an extra ``reciprocity'' assumption on the weights $Q_{ij}$ assuming that they can be represented through a symmetric matrix $C\in \R^{I\times I}$ (called the conductance matrix) with non negative elements and  $C_{ij}>0$ iff $\{i,j\}\in E$ by imposing
\begin{equation}\label{reciprocity}Q_{ij}=\frac{C_{ij}}{\sum_jC_{ij}}\,,\quad i\in I\,,\; j\in I\end{equation}
The value $C_{ij}=C_{ji}$ can be interpreted as a measure of the ``strength'' of the relation between $i$ and $j$. For two regular nodes connected by an edge, the interpretation is a sort of reciprocity in the way the nodes trust each other.
Notice that $C_{ij}$ when $i\in S$ is not used in defining the weights, but is anyhow completely determined by the symmetry assumption. Finally, the terms $Q_{ij}$ when $i,j\in S$ do not play any role in the sequel and for simplicity we can assume they are all equal to $0$.
By the definition~\eqref{reciprocity} and from matrix $C$ we are actually defining a square matrix $Q\in \real^{I\times I}$. Compactly, if we consider the diagonal matrix $D_{C\1}\in\R^{I\times I}$ defined by $(D_{C\1})_{ii}=(C\1)_{i}$, where $\1$ is all ones vector with appropriate dimension, the extension is obtained by putting $Q=D_{C\1}^{-1}C$. The matrix $Q$ is said to be a {\em time-reversible} stochastic matrix in the probability jargon.
The special case of uniform weights considered before fits in this framework, by simply choosing $C=A_G$, where $A_G$ is the adjacency matrix of the graph. In this case all edges have equal strengths and the resulting time-reversible stochastic matrix $Q$ is known as the simple random walk (SRW) on $G$.


\section{The electrical network analogy}\label{sect:electrical}
In this section we briefly recall the basic notions of electrical circuits and we illustrate the relation with our problem.
A connected graph $G=(I,E)$  together with a conductance matrix $C\in\R^{I\times I}$ can be interpreted as an electrical circuit where an edge $\{i,j\}$ has electrical conductance $C_{ij}=C_{ji}$ (and thus resistance $R_{ij}=C_{ij}^{-1}$). The pair $(G,C)$ will be called an {\it electrical network} from now on.

An {\it incidence matrix} on $G$ is any matrix
$B\in\{0,+1, -1\}^{E\times I}$ such that $B\1=0$ and $B_{ei}\neq 0$ iff $i\in e$. It is immediate to see that given $e=\{i,j\}$, the $e$-th row of $B$ has all zeroes except $B_{ei}$ and $B_{ej}$: necessarily one of them will be $+1$ and the other one $-1$ and this will be interpreted as choosing a direction in $e$ from the node corresponding to $+1$ to the one corresponding to $-1$. Define $D_C\in\R^{E\times E}$ to be the diagonal matrix such that $(D_C)_{ee}=C_{ij}=C_{ji}$ if $e=\{i,j\}$. A standard computation shows that $B^*D_CB=D_{C\1}-C$.

On the electrical network $(G,C)$ we now introduce current flows. Consider a vector $\eta\in\R^I$ such that $\eta^*\1=0$: we interpret $\eta_i$ as the input current injected at node $i$ (if negative being an outgoing current). Given $C$ and $\eta$, we can define the voltage $W\in\R^I$ and the current flow $\Phi\in \R^E$ in such a way that the usual Kirchoff and Ohm's  law are satisfied on the network. Compactly, they can be expressed as
\begin{equation*}\label{electrical relations 2} \left\{ \begin{array}{l}B^*\Phi=\eta\\ D_CBW=\Phi\end{array}\right.
\end{equation*}
Notice that $\Phi_e$ is the current flowing on edge $e$ and sign is positive iff flow is along the conventional direction individuated by $B$ on edge $e$. Coupling the two equations we obtain $(D_{C\1}-C)W=\eta$ which can be rewritten as
\begin{equation}\label{harmonic}L(Q)W=D_{C\1}^{-1}\eta\end{equation}
where $L(Q):=I-Q$ is the so called Laplacian of $Q$.
Since the graph is connected,  $L(Q)$ has rank $\card{I}-1$ and $L(Q)\1=0$. This shows that \eqref{harmonic} determines $W$ up to translations.
Notice that $(L(Q)W)_i=0$ for every $i\in I$ such that $\eta_i=0$. For this reason, in analogy with the Laplacian equation in continuous spaces, $W$ is said to be {\em harmonic} on $\{i\in I\,|\, \eta_i=0\}$. Clearly, given a subset $S\subseteq I$ and a $W\in\R^I$ which is harmonic on $S^c$, we can always interpret $W$ as a voltage with input currents given by $\eta =D_{C\1}L(Q)W$ which will necessarily be supported on $S$. $W$ is actually the only voltage harmonic on $S^c$ and with assigned values on $S$.

It is often possible to replace an electrical network by a simplified one without changing certain quantities of interest. An useful operation is {\em gluing:} if we merge vertices having the same voltage into a single one, while keeping all existing edges, voltages and currents are unchanged, because current never flows between vertices with the same voltage. Another useful operation is replacing a portion of the electrical network connecting two nodes $h,k$ by an {\em equivalent resistance}, a single resistance denoted as $\supscr{R}{eff}_{hk}$ which keeps the difference of voltage $W(h)-W(k)$ unchanged. 
Following our interpretation of Eq.~\eqref{reciprocity}, we see $(\supscr{R}{eff}_{hk})^{-1}$ as measuring of the strength of the relation between two nodes $h,k$.
Two basic cases of this operation consist in deleting degree two nodes by adding resistances (series law) and replacing multiple edges between two nodes with a single one having conductance equal to the sum of the various conductances (parallel law). These techniques will be heavily used in deriving our algorithm.

\subsection{Social networks as electrical networks}

We are now ready to state the relationship between social and electrical networks.
Consider a connected graph $G=(I,E)$, a subset of stubborn agents $S\subseteq I$, and
a stochastic time-reversible matrix $Q$ having conductance matrix $C$.
%
Notice that relation (\ref{fixed point}) can also be written as
\begin{equation}\label{harmonic3} L(Q)\left(\begin{matrix}x^{R,\ell}(\infty)\\ x^S(0)\end{matrix}\right)=\left(\begin{matrix}0\\ \theta\end{matrix}\right)\end{equation}
for some suitable vector $\theta\in\R^S$ (where $\theta$ represents the initial opinions of the stubborn agents). Comparing with (\ref{harmonic}), it follows that  $x^{R,\ell}(\infty)$ can be interpreted as the voltage at the regular agents when stubborn agents have fixed voltage $x^S(0)$ or, equivalently, when input currents $D_{C\1}\theta$ are applied to the stubborn agents. Because of equation (\ref{harmonic3}), the vector $x^{R,\ell}(\infty)$ is called the {\it harmonic} extension of $x_S(0)$ and the function $H$ defined in \eqref{H_definition} the  {\em Harmonic Influence Centrality (HIC)}.

Thanks to the electrical analogy we can compute the asymptotic opinion of the agents by computing the voltages in the graph seen as an electrical network.  From now on, we will stick to this equivalence and we will exclusively talk about an electrical network $(G,C)$ with a subset $S^{0}\subseteq I$ of nodes at voltage $0$. For any $\ell \in I\setminus S_0$, $W^{(\ell)}$ denotes the voltage on $I$ such that $W^{(\ell)}(i)=0$ for every $i\in S^{0}$ and $W^{(\ell)}({\ell})=1$.
Using this notation and the association between limiting opinions and electric voltages provided in Eqs. \eqref{harmonic} and \eqref{harmonic3}, we can express the Harmonic Influence Centrality of node $\ell$ as
\[H(\ell) = \sum_{i\in I, i\ne \ell} W^{(\ell)}(i).\]

Consider the following simple example.

\begin{example}[Line graph]\label{exa:line}
Consider the electrical network $(L, A_L)$ where $L$ is the line graph $L=(\fromto{0}{n},E_L)$. Moreover,  assume $S^0=\{0\}$ and $S^1=\{n\}$. Then, the induced voltage $W$ satisfies $W(i)=\frac{i}n$, and thus $H(n)=\frac{n+1}2$.
\end{example}

The following result, which is an immediate consequence of (\ref{harmonic3}), provides a formula that is useful in computing the voltages.
\begin{lem}[Voltage scaling]\label{lem:W-scaling}
Consider the electrical network $(G,C)$ and a subset $S^0\subseteq I$. Let $W^{(\ell)}$ be the voltage which is $0$ on $S^0$ and $1$ on $\ell$. Let $W$ be another voltage such that
$W(s)=w_0$ if $s\in S^0$ and $W({\ell})=w_1$. Then, for every node $i \in I$ it holds
\be \label{pot_pas} W(i)=w_0+(w_1-w_0)W^{(\ell)}(i)
\ee
\end{lem}

\section{The electrical analogy on trees}\label{sect:electrical-trees}
The case when the graph $G$ is a {\em tree} is very important to us, since we are able to prove a number of useful results on the solution of the OSAP problem and it will play a pivotal role in the introduction of our message passing algorithm for general graphs.
To begin with, we establish some useful notation and basic results.

In the foregoing, we assume to have fixed a tree $T=(I,E)$, a conductance matrix $C$ and the subset $S^0\subseteq I$ of $0$ voltage nodes. We next define subsets of nodes of the given tree, which enable us to isolate the effects of the upstream and downstream neighbors of a node in computing its harmonic centrality and its voltage. Given a pair of distinct nodes $i,j\in I$, we let $I^{<ij}$ denote the subset of nodes that form the subtree rooted at node $i$ that does not contain node $j$. Formally,
$$I^{<ij}:=\{h\in I\;|\; \hbox{\rm the simple path from $h$ to $j$ goes through}\; i\}.$$
We also define $I^{ij>}:=I^{<ji}$,
$ I^{ij<}:=(I^{ij>})^c\cup\{j\},$ $I^{>ij}:=I^{ji<}$, and $I^{i<j}:=I^{ij<}\cap I^{>ij}$.
Figure~\ref{fig:potentials_example} illustrates these definitions.

\begin{figure}
\centering
\includegraphics[scale=0.5]{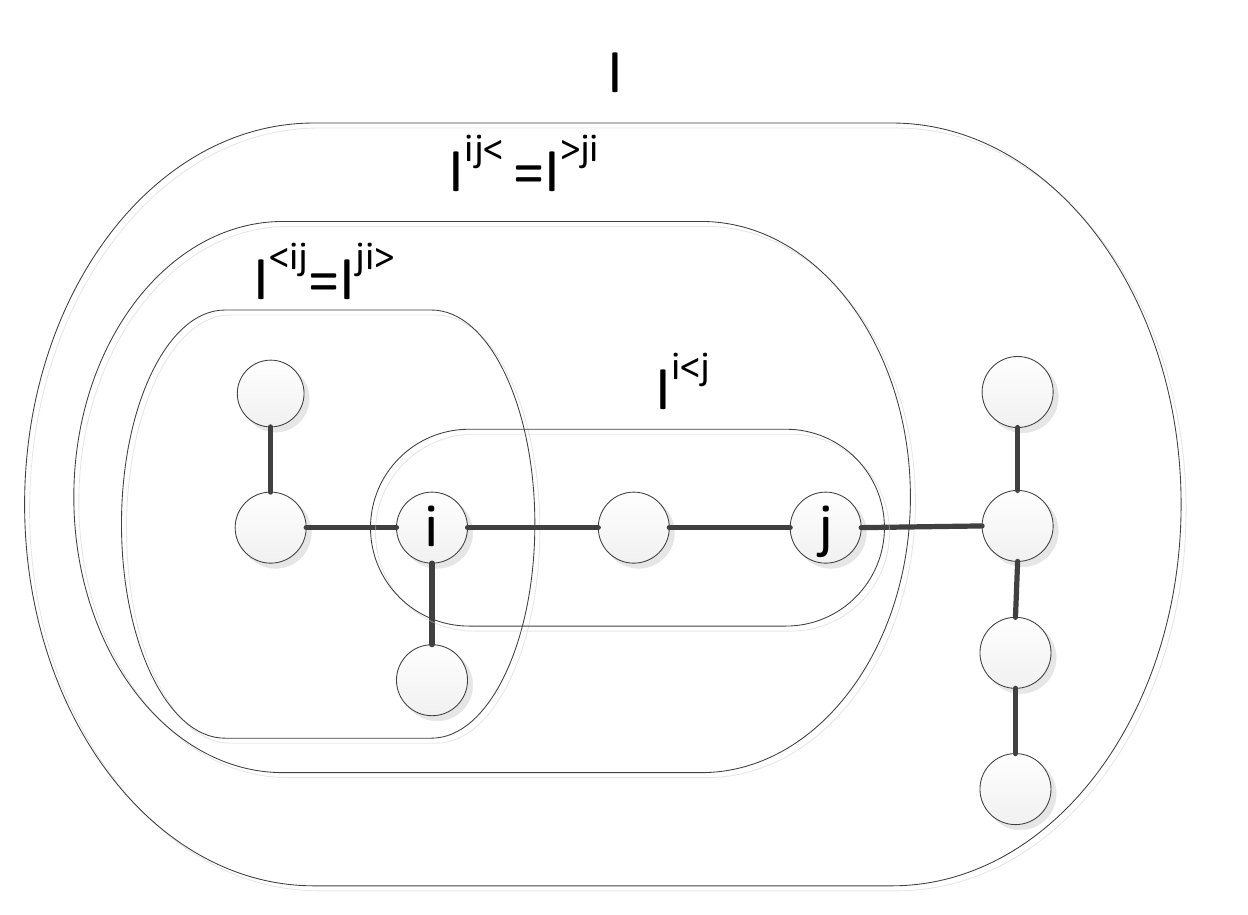}
\caption{An example of tree presenting the notation of the subsets of I}\label{fig:potentials_example}
\end{figure}
The induced subtrees is denoted using the same apex $T^{<ij}$ and so on; similarly the HIC of the nodes on each of the trees above is denoted as $H^{<ij}(\cdot )$ and so on.
Finally, we use the notation $R^{\rm eff}_{<ij}$ to denote the effective resistance inside $T^{<ij}$ between $S^0\cap I^{<ij}$ (considered as a unique collapsed node) and node $i$. We will conventionally interpret this resistance as infinite in case $S^0\cap I^{<ij}=\emptyset$. Analogously, we define $R^{\rm eff}_{ij>}:=R^{\rm eff}_{<ji}$.

Given a pair of distinct nodes $i,\,j\in I$, consider the  two voltages $W^{(i)}$ and $W^{(j)}$. If we restrict them to $T^{<ij}$, we may interpret  them as two voltages on $T^{<ij}$ which are $0$ on $S^0$ and take values in node $i$, respectively, $W^{(i)}(i)=1$ and $W^{(j)}(i)$. It follows by applying Lemma  \ref{lem:W-scaling} that
\be\label{scaling ij}W^{(j)}(\ell)=W^{(j)}(i)W^{(i)}(\ell)\quad \forall \ell\in I^{<ij}\ee
Moreover, $W^{(j)}(i)$ can be computed through effective resistances replacing the circuit determined by the subtree $T^{<ij}$, by an equivalent circuit represented by a line graph with three nodes $S^0$, $i$, and $j$ as in Figure~\ref{fig:equivalent_circuitvsd}.
\begin{figure}
\centering
\includegraphics[scale=0.5]{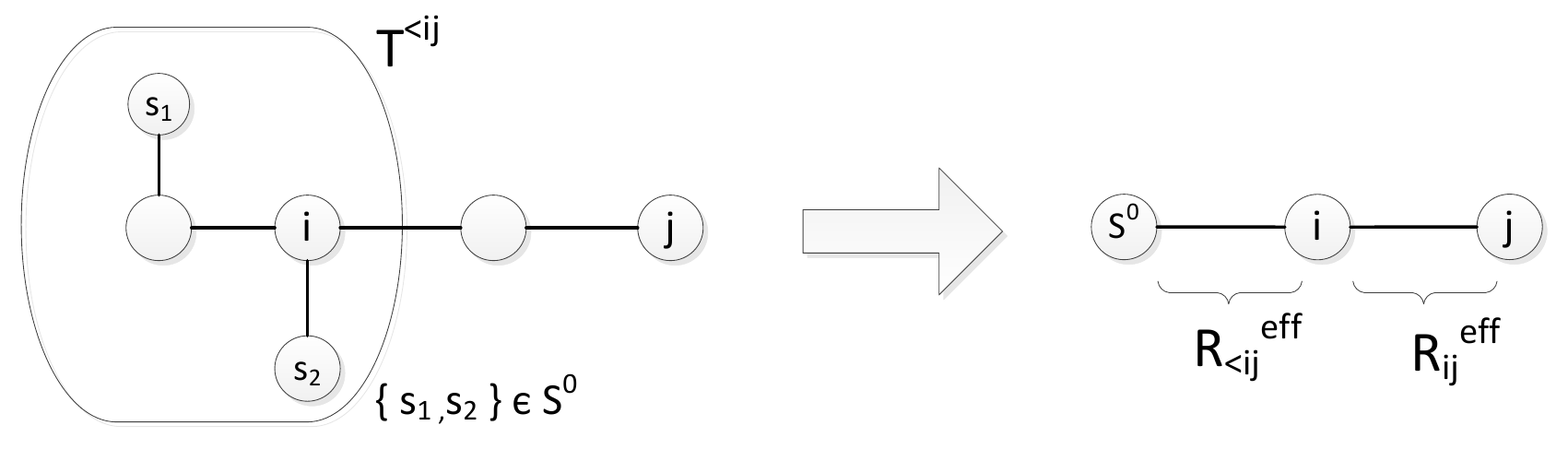}
\caption{A subtree equivalently represented as a line graph.}\label{fig:equivalent_circuitvsd}
\end{figure}
We recall that collapsing all nodes of $S^0$ in a single node is possible since they all have the same voltage. Moreover,  by definition, the edge $\{S^0,i\}$ has resistance $R^{\rm eff}_{<ij}$, while $\{i,j\}$ has resistance $R^{\rm eff}_{ij}$. Therefore, since the current flowing along the two edges is the same, Ohm's law yields
$$\frac{W^{(j)}(j)}{R^{\rm eff}_{<ij}+R^{\rm eff}_{ij}}=\frac{W^{(j)}(i)}{R^{\rm eff}_{<ij}}$$
yielding
\be\label{scaling voltage}W^{(j)}(i)=\frac{R^{\rm eff}_{<ij}}{R^{\rm eff}_{<ij}+R^{\rm eff}_{ij}}\ee
(equal to $1$ in case $S^0\cap I^{<ij}=\emptyset$).
From relations (\ref{scaling ij}) and (\ref{scaling voltage}), later on we will derive iterative equations for the computation of voltages on a tree.
%
In the rest of this section we prove some properties of the OSAP on a tree, which provide a significant simplification at the computational level.

\subsection{Nodes in $S^0$ can be assumed to be leaves}

A first  general remark to be made is that nodes in $S^0$ break the computation of the HIC into separate non-interacting components. Indeed, the induced subgraph $T_{|I\setminus S_0}$ is a forest composed of subtrees
 $\{T_h=(J_h,E_h)\}_{h\in\{1,\ldots,n\}} $.
For every ${h\in\{1,\ldots,n\}} $, define the set $S^0_h$ as the set of type 0 stubborn nodes that are adjacent  to nodes in $J_h$ in the graph $G$, that is,
$$S^0_h:=\{i \in S^0 | \exists j \in J_h: \{i,j\} \in E\}.$$
Then, define the tree $\widehat{T}_h=G|_{S^0_h \cup I_h}$, which is therefore the tree $T_h$ augmented with its type 0 stubborn neighbors in the original graph $G$.
An example of this procedure is shown in Figure~\ref{fig:derivation_example}.
\begin{figure*}
\centering
\includegraphics[scale=0.3]{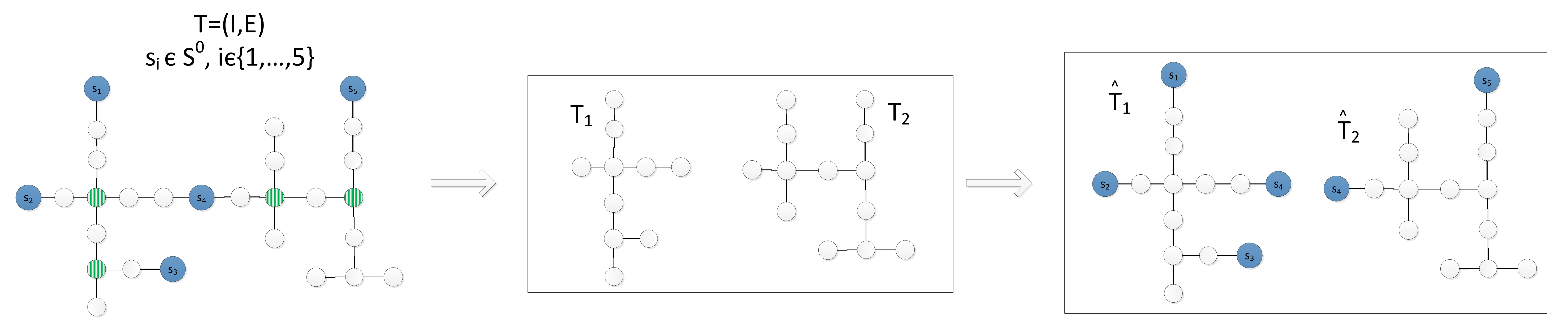}
\caption{A tree with type 0 stubborn nodes in blue, together with its decomposition according to Proposition~\ref{same_pot}.
}\label{fig:derivation_example}
\end{figure*}
The following result shows that it is sufficient to compute the HIC on the subtrees  $\widehat{T}_h$; in its statement we denote by $\widehat{H}_h$ the HIC function on $\widehat{T}_h$.
\begin{prop}[Tree decomposition]\label{same_pot} For every $h=1,\dots ,n$, and for very $\ell \in I_h$
$$\widehat{H}_h(\ell)=H(\ell) \qquad \forall \,\ell \in I_h.$$
\end{prop}
\begin{proof}
Given $\ell \in I_h$, the voltage $W^{(\ell)}(j)$ is zero for every $j$ such that the (unique) path from $j$ to $\ell$ goes through a type 0 stubborn node. This implies that $W^{(\ell)}(j)=0$ for every $j\in I\setminus I_h$.
This observation proves the result.
\end{proof}
Thanks to this property it is sufficient to study the Harmonic Influence Centrality on the subtrees $\widehat{T_h}$ and then compute
$$\max_{\ell \in I \setminus S^0} H(\ell)=\max_{h\in \{1,\ldots,n\}} \max_{\ell\in I_h} \widehat{H}(\ell).$$
Consequently, we will assume from now on that {\em stubborn agents are located in the leaves}, without any loss of generality.

\subsection{Further properties in case of unitary resistances}

In this paragraph we present a number of results characterizing the OSAP in the special, but important, case when all resistances are unitary. We recall that the social interpretation of this case is when all edges expressing relationship among regular nodes have equal strength.

Before proceeding, let us first consider the trivial situation in which there is {\em only one type 0 stubborn} (which is a leaf of $T$, without loss of generality). In this case, the maximizer of $H$ in $I$ is the node {\em adjacent} to the node in $S^0$. Indeed, let $s$ be the node in $S^0$ and $i$ its unique neighbor: then $W^{(i)}(\ell)=1$ for all $\ell\neq s$, that is $H(i)=\card{I}-1$, which is obviously the largest achievable value.
%
%
In order to deal with more than one stubborn, we need the following preliminary result.
\begin{lem}\label{lem:R_eq}
Let $T=(I,E)$ be a tree with unitary resistances and $a,b\in I$. Let $W$ be the voltage such that $W(a)=0$ and $W(b)=1$. Then, the effective resistance between $a$ and $b$
satisfies
\begin{equation*}\label{eq:R_eq}
R^{\rm eff}_{ab}\leq 2 \sum_{i \in I}W(i)-1.
\end{equation*}
\end{lem}

\begin{proof}
%
Consider the simple path $a,j_1,\dots ,j_{t-1},b$ connecting $a$ to $b$ in $T$. Clearly, by Example \ref{exa:line}, we have that $W(j_s)=s/t$. Therefore,
$$2 \sum_{i \in I}W(i)-1\geq 2 \sum_{s=1}^{t-1}W(j_s)+1
=t= R^{\rm eff}_{ab}.$$
\end{proof}
We are now ready to state and prove the main result of this section, which provides a set of necessary conditions for a node to be a solution to OSAP: {\em
provided there are at least two stubborn, the solution lies on a path connecting two of them and is a node of degree at least three, unless all paths connecting  two stubborn are made of nodes of degree two.}
\begin{thm}[Necessary conditions for OSAP] \label{propos}
Consider the OSAP problem~\eqref{eq:OSAP} over a tree $T=(I,E)$ with all unitary resistances. Assume that $S^0\subseteq I$ is a subset of leaves and that $|S^0|\geq2$. Define
\begin{align*}K=\{i \in I | \exists s',s'' \in S^0  \text{such that $i$ belongs to a simple path} \\ \text{ between $s'$ and $s''$}\}.\end{align*} Then,
 $$ \displaystyle \argmax_{\ell \in I} H(\ell) = \argmax_{\ell \in K} H(\ell).$$
Moreover, if $K':=\{ i \in K | d_i\geq3\}\neq \emptyset$, then
 $$ \argmax_{\ell \in I} H(\ell) = \argmax_{\ell \in K'} H(\ell).$$
\end{thm}
\begin{proof}
Let $j\not \in K$. Then, there exists $i\in K$ such that every path from $j$ to $S^0$ contains $i$. We claim that $H(i)>H(j)$, proving that the optimum must belong to $K$.
Harmonic influences can be computed as follows
\begin{align*}\label{322}
H(j)= &\sum_{y \in I^{<ij}\setminus\{i\}} W^{(j)}(y) + \sum_{x \in I^{i<j}} W^{(j)}(x) + |I^{ij>}|-1\\
H(i)=&\sum_{y \in I^{<ij}\setminus\{i\}} W^{(i)}(y) + |I^{<ij}|+|I^{ij>}|-1\,.
\end{align*}
We can then observe that, as a consequence of the scaling formula~\eqref{scaling ij},
\begin{equation*}\label{324}
\sum_{y \in I^{<ij}\setminus\{i\}}  W^{(j)}(y)<\sum_{y \in I^{<ij}\setminus\{i\}}  W^{(i)}(y),
\end{equation*}
while, clearly,
$\sum_{x \in I^{<ij}} W^{(j)}(x) \leq |I^{<ij}|.$
These inequalities prove the claim.

Next, we need to prove that the optimal node has degree larger than two, that is, it belongs to $K'$ (provided this set is not empty). Consider a path between two nodes in $S^0$ and a maximal string of consecutive nodes of degree two in this path, denoted as $j_1, j_2, \ldots, j_n$.
Let $a,b\in S^0\cup K'$ be the two nodes such that $\{a,j_1\}, \{j_n,b\}\in E$.
For any $x=1,\dots ,n$, consider $W^{(j_x)}$ and notice that, because of
formula~\eqref{scaling ij}, there hold
$$W^{(j_x)}({\ell})=W^{(j_x)}(a)W^{(a)}({\ell}) \;\forall \ell\in I^{<ab}$$ and $$W^{(j_x)}({\ell})=W^{(j_x)}(b)W^{(b)}({\ell}) \;\forall \ell\in I^{ab>}$$
Combining these equations with Example~\ref{exa:line}, we can compute
\begin{align*}
H(j_x)&=\!\sum_{\ell\in I^{<ab}} \!W^{(j_x)}(\ell)+\!\sum_{\ell\in I^{ab>}} \!W^{(j_x)}(\ell)+\!\sum_{\ell\in I^{j_1< j_n}}\! W^{(j_x)}(\ell)\\
&=W^{(j_x)}(a)H^{<ab}(a) + W^{(j_x)}(b) H^{ab>}(b) \\ &\quad+
  \frac{1+W_a}{2}(x-1) + \frac{1+W_b}{2}(n-x) + 1.
\end{align*}
Since, by (\ref{scaling voltage}), $W(a)=\frac{R^{\rm eff}_{<ab}}{R^{\rm eff}_{<ab}+x}$ and $W(b)=\frac{R^{\rm eff}_{ab>}}{R^{\rm eff}_{ab>}+n-x+1}$, we deduce that
\begin{align*} H(j_x)&=\frac{R^{<ab}}{R^{<ab}+x} \Big(H^{<ab}(a)+\frac{x-1}2\Big)+\frac{x-1}2\\&+\frac{R^{ab>}}{R^{ab>}+n-x+1} \Big(H^{ab>}(b)+\frac{n-x}2\Big)+ \frac{n-x}2+ 1
\end{align*}
Notice that the above expression naturally determines an extension for $x\in [1, n]\subseteq \R$
and it is straightforward to compute that
\begin{align*}
 \frac{\partial^2}{\partial x^2}H(j_x)=&
\frac{2 R^{<ab}( H^{<ab}(a)-\frac{R^{<ab}+1}2)}{(R^{<ab}+x)^2}\\&+
\frac{2 R^{ab>} (H^{ab>}(b)-\frac{R^{ab>}+1}2)}{(R^{ab>}+x)^2}.
\end{align*}
Since Lemma~\ref{lem:R_eq} implies that $2 H^{<ab}(a)\ge R^{<ab}+1$ and $2 H^{ab>}(b)\ge R^{ab>}+1$, this second derivative is nonnegative and we conclude that $H(j_x)$ is convex in $x$.
Then, we conclude that the value in $a$ or in $b$ of the HIC is greater or equal than the value on the nodes having degree two. By this statement, the proof is complete.
\end{proof}

As a consequence of the necessary conditions in Theorem~\ref{propos}, in order to solve OSAP it is sufficient to compute the HIC only on a subset of nodes, which can be much smaller than the whole node set of the social network.
Example of the effectiveness of this result are provided in Figures~\ref{fig:derivation_example} and~\ref{fig:tree_cases}, where the candidate nodes to solve problem~\eqref{eq:OSAP} are pictured with green stripes: in the former example, the candidates are reduced from $22$ to $4$.
\begin{figure*}
\centering
\includegraphics[scale=0.38]{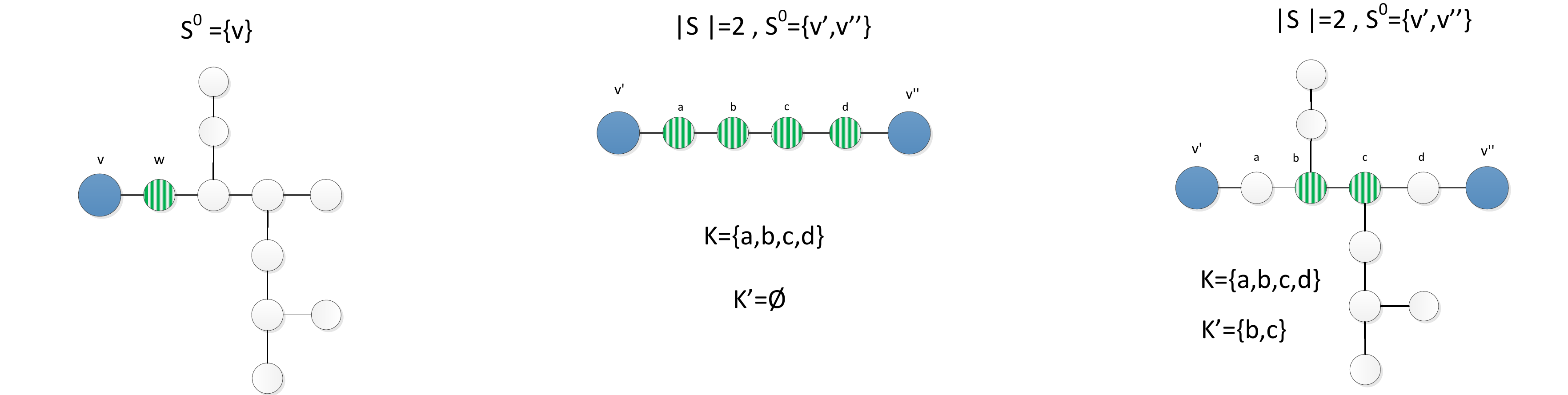}
\caption{Three different trees: blue nodes are stubborn and the ones with green stripes are the candidate solutions.}\label{fig:tree_cases}
\end{figure*}

\section{Message passing on trees}\label{sect:trees-mpa}

In this section, we design a message passing algorithm (MPA), which computes the HIC of every node of a tree in a distributed way.
We begin by outlining the structure of a general message passing algorithm on a tree. Preliminarily, define any node $i$ in the graph as the \emph{root}.
In the first phase, messages are passed inwards: starting at the leaves, each node passes a message along the unique edge towards the root node. The tree structure guarantees that it is possible to obtain messages from all other neighbor nodes before passing the message on. This process continues until the root $i$ has obtained messages from all its neighbors.
The second phase involves passing the messages back out: starting at the root, messages are passed in the reverse direction. The algorithm is completed when all leaves have received their messages.

Next, we show how this approach can be effective in our problem.
Take a generic root node $i \in I\backslash S^0$ and, for every $j\in N_i$, notice the following iterative structure of the subtree rooted in $i$ and not containing $j$:
$$I^{<ij}=\bigcup\limits_{k\in N_i\setminus\{j\}}I^{<ki}\cup \{i\}$$
This, together with relation (\ref{scaling ij}), yields
%
%
\begin{align}
\nonumber H^{<ij}(i)&=\sum\limits_{k\in N_i\setminus \{j\}}\sum_{\ell \in T^{<ki}}W^{(i)}(\ell)+1
\\
\label{iterate H} &=
\sum\limits_{k\in N_i\setminus \{j\}}W^{(i)}(k)H^{<ki}(k)+1.\end{align}
Forthermore, (\ref{scaling voltage}) yields
\be\label{voltages scale} W^{(i)}(k)=\frac{R^{\rm eff}_{<ki}}{R^{\rm eff}_{<ki}+R_{ik}}\ee
where we conventionally assume that $R^{\rm eff}_{<ki}=\infty$ and $W^{(i)}(k)=1$ if $S^0\cap I^{<ki}=\emptyset$.
On the other hand, also effective resistances $R^{\rm eff}_{<ki}$ admit an iterative representation. Indeed, replace $T^{<ij}$ with the equivalent circuit consisting of nodes: $S^0$, $i$, $j$, and all the nodes $k\in N_i\setminus\{j\}$. Between $S^0$ and $i$ there are $|N_i|-1$ parallel (length $2$) paths each passing through a different $k\in N_i\setminus\{j\}$ and having resistance $R^{\rm eff}_{<ki}+R_{ik}$ (see Figure~\ref{fig:iterative_treesvsd}).
\begin{figure*}
\centering
\includegraphics[scale=0.45]{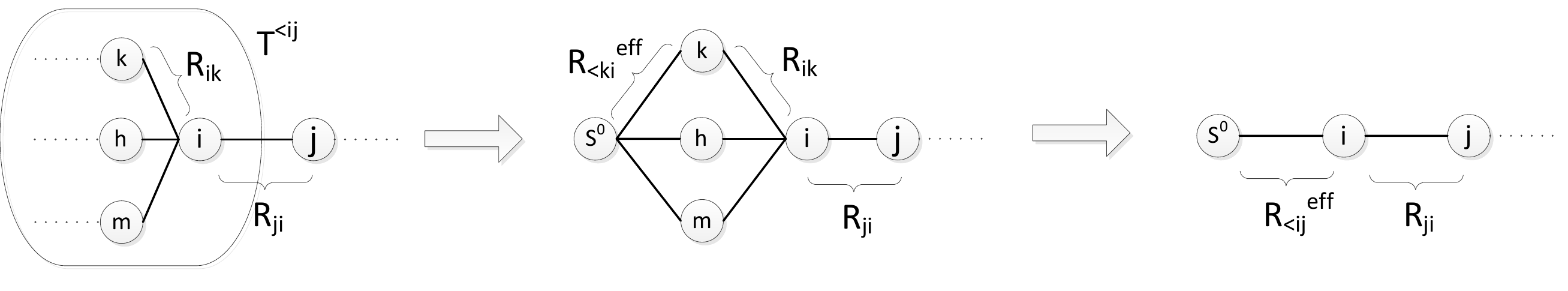}
\caption{Equivalent representation of parallel paths between $i$ and $S^0$.}\label{fig:iterative_treesvsd}
\end{figure*}
Therefore, using the parallel law for resistances we obtain
\be\label{iterate R} R^{\rm eff}_{<ij}=\left(\sum\limits_{k\in N_i\setminus \{j\}}\frac{1}{R^{\rm eff}_{<ki}+R_{ik}}\right)^{-1}\ee
The three relations (\ref{iterate H}), (\ref{voltages scale}),  and (\ref{iterate R}) determine an iterative algorithm to compute $H(i)$ at every node $i$ starting from leaves and propagating to the rest of the graph.
More precisely, define $H^{i\to j}:=H^{<ij}(i)$, and $W^{i\to j}:=W^{(i)}(j)$ to be thought as messages sent by node $i$ to node $j$ along the edge $\{i,j\}$.
From (\ref{iterate H}), (\ref{voltages scale}),  and (\ref{iterate R}), we easily obtain the following iterative relations
\be\label{algorithm1}\begin{array}{rcl} H^{i\to j}&=&\sum\limits_{k\in N_i\setminus \{j\}}W^{k\to i}H^{k\to i}+1\\
W^{i\to j}&=&\left(1+R_{ij}\sum\limits_{k\in N_i\setminus \{j\}}\displaystyle\frac{1-W^{k\to i}}{R_{ik}}\right)^{-1}\end{array}\ee
Notice that a node $i$ can only send messages to a neighbor $j$, once he has received messages $H^{k\to i}$ and $W^{k\to i}$ from all its neighbors $k$ but $j$. Iteration can start at leaves (having just one neighbor) with the following initialization step
\be\label{algorithm2}
\begin{array}{rcl} H^{i\to j}&=&\1_{i\not\in S^0}\\
W^{i\to j}&=&\1_{i\not\in S^0}\end{array}
\ee
where we denote by $\1_{i\not\in S^0}$ a vector indexed in $I$ which has entry 1 if $i\not\in S^0$ and entry 0 elsewhere.
Notice that each regular agent $i$ can finally compute $H(i)$ by the formula
\begin{equation*}\label{final compute} H(i)=\sum\limits_{k\in N_i}W^{k\to i}H^{k\to i}+1
\end{equation*}

Clearly this algorithm terminates in finite time, because as soon as a node $i$ has received all messages from its neighbors, it can compute the correct value of the Harmonic Influence Centrality.
It is then interesting to estimate the convergence time of the algorithm, in terms of number of steps, as well as the number of messages to be sent by the nodes and the number of operations to be performed.

\begin{prop}[Complexity of MPA]
Consider the message passing algorithm defined in equations \eqref{algorithm1} and \eqref{algorithm2} on tree $T$, and let $\dmax$ denote the largest degree of its nodes. Then,
\begin{enumerate}
\item the algorithm converges after at most $\diam(T)$ steps;
\item the number of messages (triples) to be sent by node~$i$ is $d_i$;
\item the number of messages (triples) to be sent in the whole network is not larger than $2 |I|$;
\item the number of operations to be performed by each node is $ O(d_i^2)$;
\item the number of operations for the whole network is $ O(\sum_{i \in I \backslash S^0}d_i^2)$.
\end{enumerate}
\end{prop}
\begin{proof}
The time before $H(i)$ converges is clearly equal to the distance between $i$ and the furthest leave; then, the first claim follows. It is immediate to observe that, if we count a triple as a message, the number of messages sent by each node is $d_i$. Then, across the whole network we have
$\sum_{i \in I} d_i=  2 |E|= 2 |I|- 2$ messages, because $T$ is a tree.
The number of operations for each node $i$ to compute a message for a neighbor $j$ is proportional to $d_i-1.$ Since node $i$ has to compute $d_i$ different messages, the computational cost for a node $i$ is of the order $O(d_i^2)$. The last claim follows by summing over all the network.
\end{proof}

A centralized algorithm to compute the HIC in any connected network was proposed in~\cite{YA11}. Being centralized, this algorithm requires full knowledge of the topology of the graph.
The computational complexity of the algorithm~\cite{YA11} is $O((|I|-|S^0|)^3).$ Since $\sum_{i \in I \backslash S^0}d_i^2 \le \card{I\setminus S^0}\, \dmax^{\,2}$,
on graphs with bounded degrees the MPA has a much smaller complexity $O(|I|-|S^0|)$ .
The drawback of MPA is its limitation to trees. In the next section, we will work towards removing this restrictive assumption.

\section{Message passing on general graphs}\label{sect:general-mpa}

Message passing algorithms are commonly designed on trees, but also implemented with some modification over general graphs. In many cases, the application is just empirical, without a proof of convergence.
We will see in this section how to apply the MPA to every graph, with suitable modifications in order to manage the new issues. Namely, we design an ``iterative'' version of the message passing algorithm of Section~\ref{sect:trees-mpa}, which can run on every network, regardless of the presence of cycles. We show that for regular graphs with unitary resistances, this algorithm converges (but not necessarily to the correct HIC values as we demonstrate next). We also present simulation results that  show the algorithm effectiveness in computing the HIC  on families of graphs with cycles.

We let the nodes send their messages at every time step, so that we denote them as
$$W^{i\to j}(t), \quad H^{i\to j}(t), \quad \text{for all } t\ge0.$$
The dynamics of messages are
%
\begin{subequations}
\begin{align}\label{G-algorithm1}
H^{i\to j}(t+1)&=\sum\limits_{k\in N_i\setminus \{j\}}\!\!W^{k\to i}(t)H^{k\to i}(t)+1
\\
W^{i\to j}(t+1)&=\left(1+R_{ij}\!\!\sum\limits_{k\in N_i\setminus \{j\}}\displaystyle\frac{1-W^{k\to i}(t)}{R_{ik}}\right)^{-1}
\end{align}
\end{subequations}
(where $R_{ij}=C_{ij}^{-1}$ are the edge resistances)
if $i \notin S^0$ and
\begin{subequations}
\begin{align}\label{G-algorithm10}
H^{i\to j}(t+1)&=0
\\
W^{i\to j}(t+1)&=0
\end{align}
\end{subequations}
otherwise.
The initialization is
\be\label{G-algorithm2}\begin{array}{rcl} H^{i\to j}(0)&=&\1_{i\not\in S^0}\\
W^{i\to j}(0)&=&\1_{i\not\in S^0}\end{array}\ee

By these definitions of messages we have defined the message passing algorithm for general graphs. Additionally, we should define a termination criterion: for instance, the algorithm may stop after a number of steps which is chosen {\it a priori}.
At every time $t$, each agent $i$ can compute an approximate $H(i)^{(t)}$ by the formula
\begin{equation*}\label{final compute} H(i)^{(t)}=\sum\limits_{k\in N_i}W^{k\to i}(t)H^{k\to i}(t)+1\end{equation*}

%
%
%
%
This new algorithm clearly converges to the HIC if the graph is a tree, and the convergence time is not larger than the diameter of the graph. Otherwise, the algorithm is not guaranteed to converge: furthermore, if the algorithm happens to converge, then the convergence value may be different from the HIC.

In order to illustrate the issues caused by the presence of cycles, we can use the so called \emph{computation trees}~\cite{W06}, which are constructed in the following way. Given a graph $G$, we focus on a `root' node, and for all $t \in \natural$ we let the nodes at distance $t$ from the root (the level $t$ of the tree) be the nodes whose messages reach the root after $t$ iterations of the message-passing algorithm.
Note that if the graph $G$ is a tree, the computation tree is just equal to $G$; otherwise, it has a number of nodes which diverges when $t$ goes to infinity. As an example, Figure~\ref{fig:computation_t} shows the first 4 levels of a sample computation tree.
In our MPA, each node $i$ is computing its own Harmonic Influence Centrality {\em in the computation tree instead than on the original graph}. As the number of levels of the computation tree diverges, the computation procedure may not converge, and --if converging-- may not converge to the harmonic influence in the original graph.
\begin{figure*}
\centering
\includegraphics[scale=0.35]{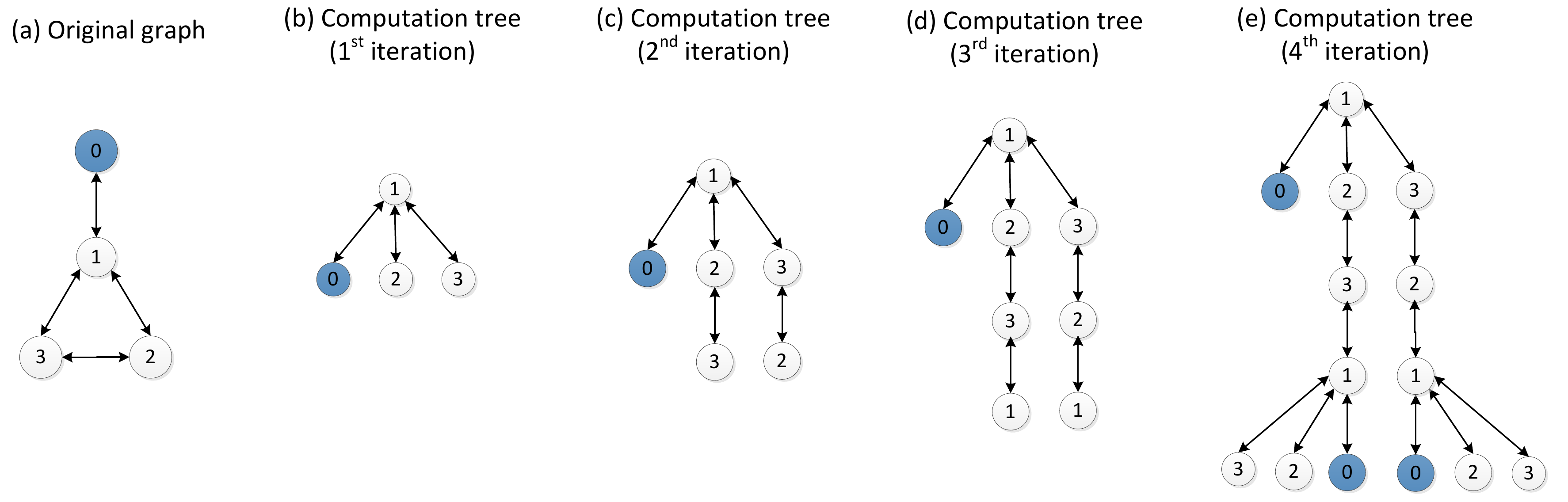}
\caption{A graph (a) and computation trees (b) of a MPA from root 1.}\label{fig:computation_t}
\end{figure*}

\subsection{MPA on regular graphs}\label{sect:mpa-regular}
This subsection is devoted to prove the following convergence result.

\begin{thm}[Convergence of MPA]\label{MPAcon}
Consider a connected graph $G=(I,E)$ with unitary resistances and $S^0\subseteq I$. Assume, moreover, that $d_i=d$ for all nodes $i\in I\setminus S^0$.
Then, the  MPA algorithm described by~\eqref{G-algorithm1}, \eqref{G-algorithm10}, and~\eqref{G-algorithm2} converges.
\end{thm}

We start analyzing the behavior of the $W$ variables which is independent from the $H$ variables. Notice that under the assumptions of Theorem~\ref{MPAcon} the second relations in \eqref{G-algorithm1}-\eqref{G-algorithm10} simplifies to
\begin{equation}\label{G-algorithm-reg} W^{i\to j}(t+1)=\left\{\begin{array}{cc}\left(d-\sum\limits_{k\in N_i\setminus \{j\}} W^{k\to i}(t) \right)^{-1}
\;\;&{\rm if}\, i\not\in S^0\\ 0\;\;&{\rm otherwise}\end{array}\right.
\end{equation}

\begin{lem}\label{W-con}
Under the assumptions of Theorem~\ref{MPAcon}, there exist numbers $W^{i\to j}\in [0,1[$  satisfying, for all $\{i,j\}\in E$, the fixed point relations
\begin{equation}\label{fixed point2} W^{i\to j}=\left\{\begin{array}{ll}\left(d-\sum\limits_{k\in N_i\setminus \{j\}} W^{k\to i} \right)^{-1}
\;\;&{\rm if}\, i\not\in S^0\\ 0\;\;&{\rm otherwise}\end{array}\right.
\end{equation}
and such that $W^{i\to j}(t)\to W^{i\to j}$ as $t\to +\infty$ for all $\{i,j\}\in E$.
\end{lem}
\begin{proof}
We already know that the sequence $W^{i\to j}(t)$ is bounded, $0\leq W^{i\to j}(t) \leq 1$  for all $\{i,j\} \in E$ and for all $t \geq 0$.
Notice now that $W^{i\to j}(1)\leq W^{i\to j}(0)$ for all $\{i,j\}\in E$. On the other hand, it is immediate to check, from expression (\ref{G-algorithm-reg}), that the following inductive step holds true:
$$W^{i\to j}(t)\leq W^{i\to j}(t-1)\;\Rightarrow\;  W^{i\to j}(t+1)\leq W^{i\to j}(t)$$
This implies that $W^{i\to j}(t)$ is a decreasing sequence for all $\{i,j\}\in E$: we thus get convergence to a limit satisfying (\ref{fixed point2}).  Finally, to show that $W^{i\to j}<1$ for all $\{i,j\}$, we notice that
$$W^{i\to j}=1\;\Rightarrow\; W^{k\to i}=1\;\forall k\in N_i\setminus\{j\}$$
Iterating this, and using the connectivity of the graph we obtain the absurd statement that $W^{h\to \ell}=1$ for $h\in S^0$ and some $\ell\in N_h$.
\end{proof}

We can say more on the limit numbers $W^{i\to j}$s.

\begin{lem}\label{W-con2} Under the assumptions of Theorem~\ref{MPAcon}, the numbers $W^{i\to j}$s defined in Lemma~\ref{W-con} satisfy the bounds
$$\sum\limits_{k\in N_i\setminus \{j\}} W^{k\to i} <1\:\forall i\in I,\; \forall j\in N_i$$
\end{lem}
\begin{proof} For every $\{i,j\}\in E$, define $\tilde W^{i\to j}=\sum_{k\in N_i\setminus \{j\}} W^{k\to i} $. From (\ref{fixed point2}), we can easily obtain the iterative relation
\begin{equation*}\label{fixed point3} \tilde W^{i\to j}=\sum\limits_{k\in (N_i\setminus\{j\})\cap (S^0)^c}\left(d-\sum\limits_{k\in N_i\setminus \{j\}} \tilde W^{k\to i} \right)^{-1}\end{equation*}
Suppose that
$$\alpha =\tilde W^{i\to j}=\max\{\tilde W^{h\to k}\,|\, \{h,k\}\in E,\; h\not\in S^0\}$$
Clearly, by Lemma~\ref{W-con} we have $\alpha\in [0, d-1[$ and we easily obtain the relation
\be\label{constraint}\alpha\leq \frac{|(N_i\setminus\{j\})\cap (S^0)^c|}{d-\alpha}\leq \frac{d-1}{d-\alpha}\ee
which yields $\alpha \in [0,1]$. Suppose by contradiction that $\alpha =1$ and notice that  inequalities in (\ref{constraint}) would yield
$$\tilde W^{i\to j}=1\;\Rightarrow\; (N_i\setminus\{j\})\cap S^0=\emptyset$$ and $$ \tilde W^{k\to i} =1\,\forall k\in N_i\setminus\{j\}.$$
Iterating this, we easily obtain the result that there is no path from nodes in $S^0$ to $i$ and this contradicts the fact that $G$ is connected.
\end{proof}

Finally, we analyze the behavior of the sequences  $H^{i\to j}(t)$.
\begin{prop}\label{H-con}
Under the assumptions of Theorem~\ref{MPAcon}, there exist numbers $H^{i\to j}\in [0,1[$ satisfying, for all $\{i,j\}\in E$, the fixed point relations
\begin{equation*}\label{fixed point3} H^{i\to j}=\left\{\begin{array}{ll}\sum\limits_{k\in N_i\setminus \{j\}}W^{k\to i}H^{k\to i}+1
\;\;&{\rm if}\, i\not\in S^0\\ 0\;\;&{\rm otherwise}\end{array}\right.\end{equation*}
and such that $H^{i\to j}(t)\to H^{i\to j}$ as $t\to +\infty$ for all $\{i,j\}\in E$.
\end{prop}
\begin{proof}
It is convenient to gather the sequences $H^{i\to j}(t)$ into a vector sequence $H(t)\in\R^E$ and rewrite the iterative relation contained in (\ref{G-algorithm1}) as
$$H(t+1)=W(t)H(t)+\1_{(S^0)^c}$$
where $W(t)\in \R^{E\times E} $ is given by
$$W(t)_{i\to j, h\to k}:=\left\{\begin{array}{ll} W^{h\to k}(t)\quad &{\rm if}\; k=i\not\in S^0,\, h\neq j\\ 0\quad  &{\rm otherwise}\end{array}\right.$$
We know from Lemmas~\ref{W-con} and~\ref{W-con2} that $W(t)$ converges to a matrix $W\in \R^{E\times E}$ with non-negative elements and satisfying the row relations
$$\sum\limits_{\{h, k\}\in E}W_{i\to j, h\to k}=\left\{\begin{array}{ll}\sum\limits_{h\in N_i\setminus\{j\}}W^{h\to i}\quad &{\rm if}\; \not\in S^0\\  0\quad &{\rm otherwise}\end{array}\right.
$$
Notice that $W$ is an asymptotically stable sub-stochastic matrix such that
$$||W||_{\infty}=\max\{\sum\limits_{\{h, k\}\in E}W_{i\to j, h\to k}\,|\, \{i,j\}\in E\}<1$$ Straightforward calculus considerations then yield convergence of $H(t)$.
\end{proof}


\subsection{Simulations}\label{sect:simulations}

We have performed extensive simulations of our MP algorithm on well-known families of random graphs such as Erd\H{o}s-R\'enyi and Watts-Strogatz, obtaining very encouraging results.
First, the algorithm is convergent in every test. Second, in many cases the computed values of HIC are very close to the correct values, which we can obtain by the benchmark algorithm in~\cite{YA11}.
In order to make this claim more precise, we define two notions of error. We denote by $H(i)$ the correct Harmonic Influence Centrality of node $i$ and by $\widehat{H}^{(t)}(i)$ the output of the algorithm in node $i$ after $t$ steps.
We define the \emph{mean deviation error} at time step $t$ as:
$$\subscr{e}{dev}(t)=\frac{\sum_{i \in I} |H(i)-\widehat{H}^{(t)}(i)|}{|I|}.$$
Additionally,  as we are interested in the optimal stubborn agent placement problem, we are specially concerned about obtaining the right ranking of the nodes, in terms of HIC. We thus define the \emph{mean rank error} at time step $t$ as
$$\subscr{e}{rank}(t)=\frac{\sum_{i \in I} |\textup{rank}_{H}(i)-\textup{rank}_{\widehat{H}^{(t)}}(i)|}{|I|},$$
where for a function $\map{f}{I}{\real}$, we denote by $\textup{rank}_{f}(i)$ the position of $i$ in the vector of nodes, sorted according to the values of $f$.

We now move on to describe some simulation results in more detail.
As the stopping criterion, we ask that the mean difference between the output of two consecutive steps is below a threshold, chosen as $10^{-5}$.
As the topology of the graphs, we choose random graphs generated according to the well-known Erd\H{o}s-R\'enyi model: $\mathcal{G}(n, p)$ is a graph with $n$ nodes such that every pair of nodes is independently included in the edge set with probability $p$ (for further details see~\cite{AB02}). 

First, we consider an example of Erd\H{o}s-R\'enyi random graph with $n=15$ and $p=0.2$; nodes 1, 2 and 3 are stubborn in $S^0$. In spite of the presence of several cycles in the sampled graph, the algorithm finds the maximum of the HIC correctly; see Figure~\ref{fig:erdosA}. Figure~\ref{fig:erdosB} plots the mean deviation error and the mean ranking error as functions of time steps in the same experiment: after just 4 steps the ranking error reaches a minimum value. Note that although the obtained ranking is not entirely correct, the three nodes with highest HIC are identified and the HIC profile is well approximated.
\begin{figure}
\centering
\includegraphics[scale=0.4]{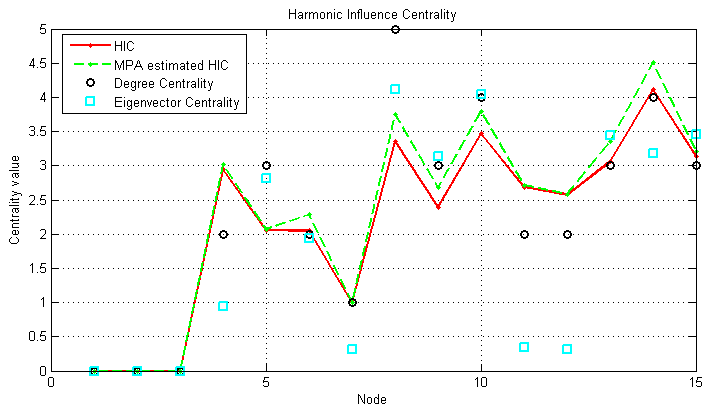}
\caption{Comparison between the actual values of $H$ in the nodes of an Erd\H{o}s-R\'enyi graph with $15$ nodes and the values estimated by the MPA. Degree and eigenvector centralities are also shown. }
\label{fig:erdosA}
\end{figure}
\begin{figure}
\centering
\includegraphics[scale=0.4]{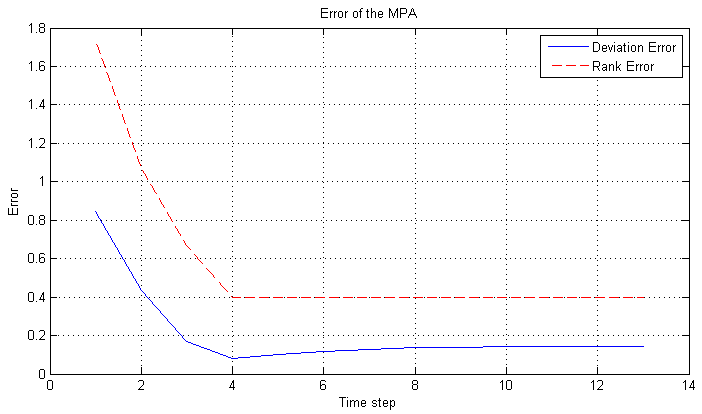}
\caption{Mean deviation error and mean ranking error of the MPA as functions of time steps on the same graph as in Fig.~\ref{fig:erdosA}. The stopping condition is reached after 13 steps.}
\label{fig:erdosB}
\end{figure}
The true HIC is smaller than the approximated one and the deviation error is localized on some node. These facts, which can be widely observed in our simulations, can be explained by thinking of the computation trees as in Section~\ref{sect:general-mpa}. Indeed, for each node $i$ the MP algorithm actually computes the HIC on the computation tree rooted in $i$. The computation tree is closer to the actual graph when the node $i$ is farther from cycles or it belongs to fewer of them; then also the computed HIC will be closer to the true one. Moreover when the number of iterations grows the computation tree has an increasing number of nodes: even if they are far from the node considered, they contribute to overestimate its centrality.

Next, we present simulations on larger Erd\H{o}s-R\'enyi random graphs: we let  $n=500$ and consider (A) $p=\ln(n)/n\approx 0.012$ and (B) $p=0.1$.
It is known~\cite{AB02} that in the regime of case (A), the graph is guaranteed to be connected for large $n$, while in case (B) the graph, besides being connected, has many more cycles and its diameter is smaller. We plot the time evolution of the error for cases (A) and (B) in Figures~\ref{500A} and~\ref{500B}, respectively.
In both cases the node with the higher Harmonic Influence Centrality has been correctly identified by the MPA, and the mean ranking error is below 3. 

\begin{figure}
\centering
\includegraphics[scale=0.4]{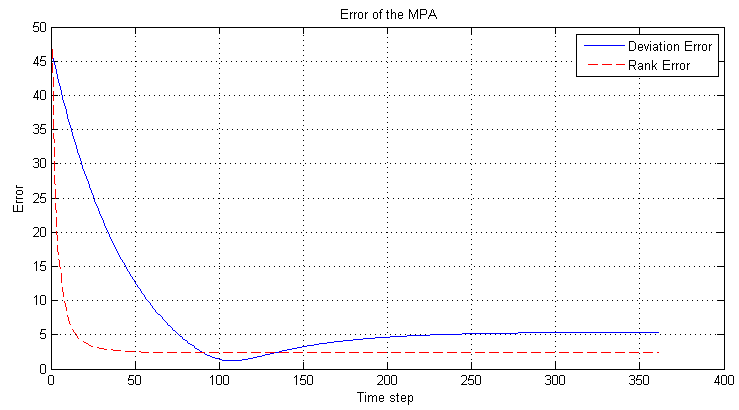}
\caption{Mean deviation error and mean ranking error of the MPA as functions of time steps on a large Erd\H{o}s-R\'enyi graph with low connectivity.}
\label{500A}
\end{figure}
\begin{figure}
\centering
\includegraphics[scale=0.4]{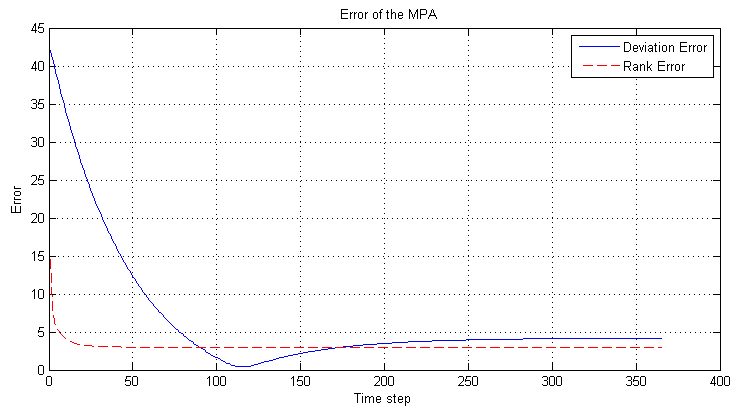}
\caption{Mean deviation error and mean ranking error of the MPA as functions of time steps on a large Erd\H{o}s-R\'enyi graph with high connectivity.}
\label{500B}
\end{figure}

In all these examples the deviation error first decreases and then increases as a function of time. This observation corresponds to the evolution of the computation trees: first they grow to represent relatively well the neighbourhood of the neighbors, but after a certain number of steps too many ``ghost'' nodes are included, thus worsening the computed approximation.


In order to stress the accuracy of our results we can compare the HIC computed through the MPA to other centralities, which can be computed in a distributed way and may be thought of as reasonable approximations of the HIC.
The first naive option is the degree centrality, that is, the number of neighbors. This is exactly what the MPA computes after one time step and the results in terms of deviation and rank error can be read on Figures~\ref{fig:erdosB},~\ref{500A} and~\ref{500B}. The experiments indicate that degrees are in general insufficient to describe the HIC and can at best be considered as rough approximations.
A second option is the eigenvector centrality: our experiments show it to be an unreliable approximation of the HIC, as it gives fairly large mean rank errors of $1.9$, $18.7$ and $9.1$, respectively, for the three experiments shown before. Similar observations on degree and eigenvector centralities can be drawn from Figure~\ref{fig:erdosA}, which includes the node-by-node values of degree and eigenvector centralities (the latter is re-scaled by the maximum of the true HIC).

The reason why these measures are inadequate to our problem is the following: both the degree centrality and the eigenvector centrality evaluate the influence of a node within a network, but they do not consider the different role of stubborn nodes, treating them as normal nodes.

\section{Conclusions}\label{sect:conclusions}

In this paper we have proposed a centrality measure on graphs related to the consensus dynamics in the presence of stubborn agents, the Harmonic Influence Centrality: this definition of centrality quantifies the influence of a node on the opinion of the global network. Although our setting assumes all stubborn except one to have the same value, the approach can be extended to more complex configurations, as discussed in~\cite{YA11}.
Thanks to an intuitive analogy with electrical networks, which holds true for time-reversible dynamics, we have obtained several properties of HIC on trees. As an application of these results, have proposed a message passing algorithm to compute the node which maximizes centrality. We have proved the algorithm to be exact on trees and to converge on any regular graph. Furthermore, numerical simulations show the good performance of the algorithm beyond the theoretical results. Further research should be devoted to extend the analysis of the algorithm beyond the scope of the current assumptions to include general networks with cycles, varied degree distributions, and directed edges.
\bibliographystyle{plain}
\bibliography{aliases,centrality-biblio}

\end{document}